\documentclass[11pt]{amsart}
\pdfoutput=1

\usepackage{amssymb}
\usepackage{amsthm}
\usepackage{amsfonts}
\usepackage{amsmath}
\usepackage{pmboxdraw}
\usepackage{verbatim} 
\usepackage{graphicx}
\usepackage{color}
\usepackage[colorlinks=true, citecolor=blue, filecolor=black, linkcolor=black, urlcolor=black]{hyperref}
\usepackage{cite}
\usepackage[normalem]{ulem}
\usepackage{subcaption}
\usepackage{todonotes}
\usepackage{mathtools}
\usepackage{enumitem}
\mathtoolsset{showonlyrefs}

\newcommand{\RN}[1]{%
	\textup{\uppercase\expandafter{\romannumeral#1}}%
}

\def\bp{{\bar\partial}}

\def\pa{\partial}

\def\C{\mathbb{C}}

\def\P{\mathbf{P}}
\def\R{\mathbb{R}}

\newcommand{\bfR}{\mathbf{R}}

\newcommand{\calK}{{\mathcal K}}

\newcommand{\bfK}{{\mathbf K}}

\newcommand{\erfc}{\operatorname{erfc}}

\newcommand{\re}{\operatorname{Re}}
\newcommand{\im}{\operatorname{Im}}

\newcommand{\Ham}{\mathbf{H}}

\newcommand{\Prob}{{\mathbf{P}}}

\renewcommand{\d}{{\partial}}

\newcommand{\Lap}{\Delta}

\newcommand{\1}{\mathbf{1}}

\newcommand{\para}{\mathbf{c}}
\newcommand{\point}{\alpha}

\theoremstyle{plain}
\newtheorem*{thm*}{Theorem}
\newtheorem{thm}{Theorem}[section]
\newtheorem{lem}[thm]{Lemma}
\newtheorem{cor}[thm]{Corollary}

\theoremstyle{definition}
\newtheorem*{eg*}{Example}
\newtheorem*{egs*}{Examples}
\newtheorem*{def*}{Definition}

\theoremstyle{remark}
\newtheorem*{rmk*}{Remark}
\newtheorem*{rmks*}{Remarks}

\numberwithin{equation}{section}

\begin{document}
\title[Random normal matrices in the almost-circular regime]{Random normal matrices in the almost-circular regime}


\author{Sung-Soo Byun}
\address{School of Mathematics, Korea Institute for Advanced Study, 85 Hoegiro, Dongdaemun-gu, Seoul 02455, Republic of Korea}
\email{sungsoobyun@kias.re.kr}

\author{Seong-Mi Seo}
\address{Department of Mathematical Sciences, Korea Advanced Institute of Science and Technology, Daejeon, 34141, Republic of Korea}
\email{seongmi@kaist.ac.kr}
 

\thanks{Sung-Soo Byun was partially supported by Samsung Science and Technology Foundation (SSTF-BA1401-51). Seong-Mi Seo was partially supported by the National Research Foundation of Korea (2019R1A5A1028324, 2019R1F1A1058006). }

\begin{abstract}
We study random normal matrix models whose eigenvalues tend to be distributed within a narrow ``band'' around the unit circle of width proportional to $\frac1n$, where $n$ is the size of matrices.
For general radially symmetric potentials with various boundary conditions, we derive the scaling limits of the correlation functions, some of which appear in the previous literature notably in the context of almost-Hermitian random matrices. 
We also obtain that fluctuations of the maximal and minimal modulus of the ensembles follow the Gumbel or exponential law depending on the boundary conditions.   
\end{abstract}
\maketitle

\section{Introduction} \label{Section_Intro}

In non-Hermitian random matrix theory, one often encounters annular domains in the limiting spectral distribution of certain models. 
A typical example of such a model is the induced Ginibre ensemble \cite{MR2881072}, an extension of the Ginibre ensemble to include zero eigenvalues.
This model is described with two parameters $a_n$ and $b_n$ $(a_n > b_n)$, which possibly depend on the size $n$ of the matrix. 
Here, the former parameter $a_n$ is related to the scaling factor and the latter one $b_n$ is related to the zero modes of the model.
(See \eqref{Gibbs} and \eqref{Qn induced} below for a precise description.)
Then the eigenvalues $\{\zeta_j\}_1^n$ of the induced Ginibre ensemble tend to be uniformly distributed on the annulus 
\begin{equation} \label{Sn induced}
S_n^I:=\Big\{ \zeta \in \C: \sqrt{ \frac{b_n}{a_n} } \le |\zeta| \le \sqrt{ \frac{b_n+1}{a_n} } \Big\}. 
\end{equation}
In particular, with some specific choices of $a_n$ and $b_n$, the annulus $S_n^I$ forms a narrow band around the unit circle of width proportional to $\frac1n$, see Figure~\ref{Fig_AC}.   
In what follows, we refer to this situation as the \emph{almost-circular regime} that will be studied in a more general context. 

\begin{figure}[h!]
	\begin{center}
		\includegraphics[width=0.7\textwidth]{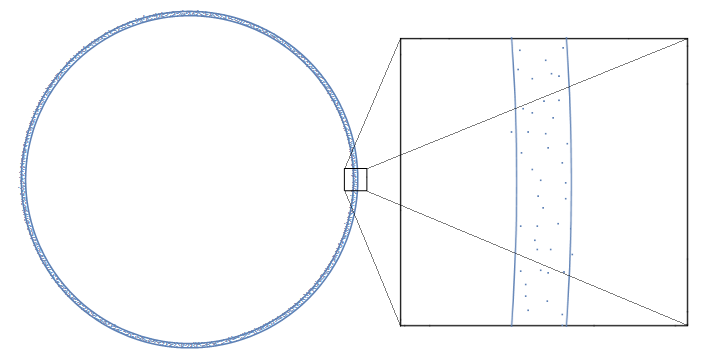}	
	\end{center}
	\caption{Eigenvalues of an induced Ginibre matrix in the almost-circular regime. 
	}
	\label{Fig_AC}
\end{figure}

From the microscopic point of view, such almost-circular ensembles can be viewed as point processes mainly distributed on a strip. 
As this is a characteristic feature of almost-Hermitian random matrices \cite{fyodorov1997almost} (or random normal matrices near a cusp type singularity \cite{MR4030288}), one can expect from the general universality principles in random matrix theory that almost-circular ensembles lie in the same universality class previously appearing in these models. 
Furthermore, contrary to almost-Hermitian random matrices, where only a few specific models were analysed (see e.g. \cite{akemann2016universality,AB21,bender2010edge,osborn2004universal,akemann2010interpolation}), almost-circular ensembles provide concrete models allowing explicit asymptotic analysis possible.

In this note, we aim to derive various scaling limits of almost-circular ensembles, which give rise to further universality classes beyond the well-known deformed sine kernel introduced by Fyodorov, Khoruzhenko, and Sommers \cite{MR1431718,MR1634312,fyodorov1997almost}. 
To be more precise, we shall consider the ensemble \eqref{Gibbs} with a general radially symmetric potential $Q_n$ and derive scaling limits under various boundary confinements. 

Let us first introduce almost-circular ensembles.
We consider a configuration $\{\zeta_j\}_1^n$ of points in $\C$ whose joint probability distribution $\Prob_n$ is of the form 
\begin{equation} \label{Gibbs}
 d \P_n =\frac{1}{Z_n} e^{-\Ham_n } \, dA_n, \qquad \Ham_n:= \sum_{j >k} \log \frac{1}{|\zeta_j-\zeta_k|^2}+n \sum_{j=1}^n  Q_n(\zeta_j),
\end{equation}
where $dA_n:=(dA)^{\otimes n}$ is the normalised volume measure on $\C^n$ with $dA(\zeta)=\frac{1}{\pi}d^2\zeta$, and $Z_n$ is the partition function which turns $\P_n$ into a probability measure.
Here, $Q_n: \C \to \R \cup \{\infty\}$ is a suitable function called external potential.

We remark that the induced Ginibre ensemble is associated with the potential
\begin{equation}\label{Qn induced}
Q_n^I(\zeta)= a_n |\zeta|^2-2b_n\log |\zeta|.
\end{equation}
This is indeed a special case of a more general Mittag-Leffler potentials, see e.g.  \cite{byun2021lemniscate,charlier2021large,ameur2018random} and references therein. 
The Mittag-Leffler ensembles are contained in the class of models under consideration in this note. 

Now we describe precise assumptions on the potentials $Q_n$. 
We begin with standard assumptions from the logarithmic potential theory \cite{ST}. 
By definition, the external potential $Q_n$ is a lower semicontinuous function which is finite on some set of positive capacity and satisfies
\begin{equation}\label{uniform}
\liminf_{\zeta\to\infty}\frac {Q_n(\zeta)}{2\log|\zeta|}\ge
k,
\end{equation}
where $k>1$ is a fixed constant independent of $n$. 
Then there exists a probability measure $\sigma_{Q_n}$ that minimises the energy functional
\begin{equation}\label{Ilog}
I_{Q_n}[\mu]=\iint_{\C^2} \log\frac 1 {|\zeta-\eta|}\, d\mu(\zeta)\, d\mu(\eta)+\int_\C Q_n \,d\mu. 
\end{equation}
The support $S_n$ of the equilibrium measure $\sigma_{Q_n}$ is called the \textit{droplet}, see e.g. \cite{HM13}.
It is also well known that if $\sigma_{Q_n}$ is absolutely continuous with respect to the area measure $dA$, the density on $S_n$ is $\Delta Q_n$, where $\Lap = \frac{1}{4}(\d_x^2 + \d_y^2)$ is the quarter of the usual Laplacian.

Next, let us set the stage to investigate almost-circular ensembles. 
We consider a radially symmetric potential 
\begin{equation} \label{Q radially symmetric}
Q_n(\zeta)=g_n(r), \qquad r=|\zeta|, \qquad g_n: \R_+ \to \R,
\end{equation} 
where $g_n$ is $C^3$-differentiable on $\R_+$. Without loss of generality, we shall assume that $g_n(1)=0$ and  $g_n'(1)=1$. 
We also assume that $Q_n$ is subharmonic in $\C$ and  strictly subharmonic in a neighbourhood of the unit circle.
Then the associated droplet $S_n$ is given by the annulus
\begin{equation} \label{S Qn}
S_n:=\{ \zeta \in \C \, :  \, r_0 \le |\zeta| \le r_1 \},
\end{equation}
where $r_0 \equiv r_{n,0}$ and $r_1 \equiv r_{n,1}$ are the pair of constants satisfying 
\begin{equation} \label{r0 r1 eq}
r_0\,g_n'(r_0)=0, \qquad r_1 \,g_n'(r_1)=2,
\end{equation}
see \cite[Section IV.6]{ST}. 
By  \eqref{Qn induced} and \eqref{r0 r1 eq}, one can easily see that the droplet \eqref{Sn induced} of the induced Ginibre ensemble is a special case of \eqref{S Qn}. 

We assume that $|g_n'''| \leq C n$ in a neighbourhood of the unit circle for some constant $C$, and the limit 
\begin{equation} \label{rho}
\rho:=\lim_{n \to \infty} \sqrt{\frac{n}{\Lap Q_n(1)}} \in (0,\infty)
\end{equation}
exists. 
Then it is easy to see that
\begin{equation} \label{r0 r1}
r_0=1-\frac{\rho^2}{4n}+o(n^{-1}), \qquad r_1=1+\frac{\rho^2}{4n}+o(n^{-1}).
\end{equation}
This behaviour \eqref{r0 r1} makes the droplet \eqref{S Qn} close to the unit circle.
Furthermore, by \eqref{r0 r1}, one can notice that the parameter $\rho$ defined by \eqref{rho} plays the role in describing the width of the droplet, see \cite{AB21} for more about the geometric meaning of such a parameter. 

Now we introduce scaling limits of almost-circular ensembles.
For this purpose, we denote by
\begin{equation} \label{bfRnk}
	\bfR_{n,k}(\zeta_1,\cdots,\zeta_k):=\frac{1}{Z_n} \frac{n!}{(n-k)!} \int_{ \C^{n-k} }  e^{-\Ham_n }  \prod_{j=k+1}^{n} dA(\zeta_j)
\end{equation}
the $k$-point correlation function of model \eqref{Gibbs}.
In particular we write $\bfR_n \equiv \bfR_{n,1}$ for the $1$-point function. 
It is well known that $\{\zeta_j\}_{j=1}^{n}$ forms a determinantal point process \cite[Chapter 5]{forrester2010log}, namely, the $k$-point function \eqref{bfRnk} is expressed as 
\begin{equation} \label{bfKn det}
\bfR_{n,k}(\zeta_1,\cdots,\zeta_k)=\det \Big(  \bfK_n(\zeta_j,\zeta_l) \Big)_{ j,l =1 }^k,
\end{equation}
where a Hermitian function $\bfK_n$ is called a correlation kernel. 
Furthermore, $\bfK_n$ can be written in terms of the reproducing kernel for the space of analytic polynomials in $L^2(e^{-nQ_n})\,dA$:
\begin{equation} \label{bfKn ONP}
    \bfK_n(\zeta,\eta) = \sum_{j=0}^{n-1} p_{n,j}(\zeta)\overline{p_{n,j}(\eta)} e^{-\frac{n}{2} (Q_n(\zeta) + Q_n(\eta))  },
\end{equation}
where $p_{n,j}$ is the orthonormal polynomial of degree $j$ with respect to the measure $e^{-nQ_n}\,dA$.
(We refer to \cite{hedenmalm2017planar} for a recent development of the theory of planar orthogonal polynomials.)

For the local investigation, we consider a zooming point $\point_n$ and define the rescaled ensemble $\{ z_j \}_1^n$ by 
\begin{equation} \label{z zeta}
	z_j:=\sqrt{n\Lap Q_n(\point_n)} \cdot (\zeta_j-\point_n),
\end{equation}
see Figure~\ref{Fig_AC} for an illustration. 
Here, the specific choice of the rescaling factor $\sqrt{n\Lap Q_n(\point_n)} $ follows from the mean eigenvalue spacing at the zooming point $\point_n.$
In the sequel, we frequently use the notation \eqref{z zeta}.
We denote the $k$-point function $R_{n,k}$ of $\{ z_j \}_1^n$ by 
\begin{equation}\label{R_n}
R_{n,k}(z):=\frac{1}{(n\Lap Q_n(\point_n))^k} \bfR_{n,k}(\zeta), \qquad z=\sqrt{n\Lap Q_n(\point_n)} \cdot (\zeta-\point_n). 
\end{equation}
By \eqref{bfKn det}, one can observe that $R_{n,k}$ enjoys the determinatnal structure as well.
In the sequel, we also write $R_n \equiv R_{n,1}.$

The following theorem is an immediate consequence of \cite[Thoerem 3.8]{AB21}. 
This approach is based on Ward's equation, see Subsection~\ref{Subsec_Ward} for more about this method.

\begin{thm} \label{thm:onelimit_fH}
Let $\point_n = 1$. Then as $n\to \infty$, $R_n$ converges locally uniformly to the limit 
\begin{equation} \label{R free}
R^{(1)}(z) = \frac{1}{\sqrt{2\pi}} \int_{-\frac{\rho}{2}}^{\frac{\rho}{2}} e^{-\frac{1}{2}(2\re z - \xi)^2} \, d\xi.
\end{equation}
\end{thm}

We emphasise that the scaling limit \eqref{R free} agrees with the one appearing in the context of almost-Hermitan random matrices \cite{MR1431718,MR1634312,fyodorov1997almost}. 
We also refer to \cite{akemann2016universality,AB21} and references therein for some known universality results in this class.
(Cf. the symplectic version of Theorem~\ref{thm:onelimit_fH} will appear in the forthcoming work \cite{byun2022symplectic}.)

Note that we intentionally add the superscript $(1)$ in \eqref{R free} since Theorem~\ref{thm:onelimit_fH} can be realised as a special case of our results described in the following section. 
This will be clarified below.
For the graph of $R^{(1)}$, see Figure~\ref{Fig_R free hard} (A).

The limiting $1$-point function determines a unique determinantal point field, see \cite[Lemma 1]{MR4030288}. 
We also refer to \cite[Theorem 1.1]{MR3975882} and \cite[Lemma 3.5]{AB21} for the structure of a limiting correlation kernel of the rescaled system. 
To be more precise, for a class of potentials which are strictly subharmonic near the droplet, a limiting correlation kernel $K$ of a properly rescaled system is of the form 
\begin{equation} \label{K structure}
K(z,w) = G(z,w)\cdot L(z,w), \qquad G(z,w) = e^{z\bar{w} - \frac{1}{2}|z|^2 - \frac{1}{2}|w|^2},
\end{equation}
where $L$ is Hermitian and analytic in $z$ and $\bar{w}$.  
Here $G$ is the Ginibre kernel, cf. \cite{forrester2010log}. 
Then as a consequence of \eqref{K structure}, one can easily see the locally uniform convergence
\begin{equation} \label{K FKS}
R_{n,k} \to \det \Big(K^{(1)}(z_j,z_l)\Big)_{j,l=1}^{k}, \qquad K^{(1)}(z,w) =   G(z,w)\,\frac{1}{\sqrt{2\pi}} \int_{-\frac{\rho}{2}}^{\frac{\rho}{2}} e^{-\frac{1}{2}(z+\bar{w} - \xi)^2} \, d\xi .
\end{equation}
We remark that the kernel $K^{(1)}$ in \eqref{K FKS} interpolates well-known Ginibre and sine kernels, see e.g. \cite[Remark 4 (c)]{akemann2016universality}.
In the following section, we discuss several generalisations of Theorem~\ref{thm:onelimit_fH}.

\section{Discussions of main results}

In this section, we introduce our model and state main results. 
Throughout this note, we shall keep assumptions in Section~\ref{Section_Intro} of the potential $Q_n$. 

From now on, we refer to the situation in Theorem~\ref{thm:onelimit_fH} where there is no boundary constraint as the \emph{free boundary} or \emph{soft edge} condition.
On the other hand, if we completely confine the gas to the droplet $S_n$, we refer to this situation as the \emph{soft/hard edge condition}. 
Equivalently, we redefine the potential $Q_n$ by setting $Q_n(\zeta)=+\infty$ when $\zeta\not\in S_n$.
(See \cite{MR4030288,MR2921180,byun2021wronskian} and references therein for previous works studying such a boundary condition.)
Indeed this terminology comes from its Hermitian counterpart, ``soft edge meets hard edge'' situation \cite{claeys2008universality}. 

In \cite{MR4030288}, it was conjectured that for the random normal matrix ensemble in the strip with soft/hard edge condition, the limiting $1$-point function $R^{(\infty)}$ is of the form 
\begin{equation}\label{R softhard}
    R^{(\infty)}(z) = \1_{\{|\re z|<\frac{\rho}{4}\}}\cdot \int_{-\frac{\rho}{2}}^{\frac{\rho}{2}} \frac{e^{-\frac{1}{2}(2\re z - \xi)^2}}{\int_{-\frac{\rho}{2}}^{\frac{\rho}{2}} e^{-\frac{1}{2}(t-\xi)^2} \, dt }\,  d\xi.
\end{equation}
(See Figure~\ref{Fig_R free hard} (B) for the graph of $R^{(\infty)}$.)
We also refer to \cite[Subsection 8.3]{AB21} for further supports of this conjecture.

\begin{figure}[h!]
    	\begin{subfigure}{0.32\textwidth}
		\begin{center}	
			\includegraphics[width=\textwidth]{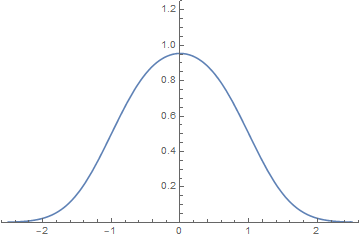}
		\end{center}
		\subcaption{free boundary}
	\end{subfigure}	
	\begin{subfigure}{0.32\textwidth}
		\begin{center}	
			\includegraphics[width=\textwidth]{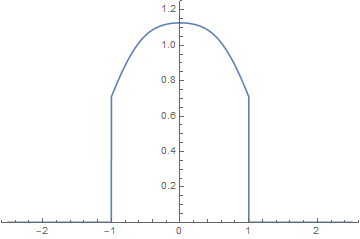}
		\end{center}
		\subcaption{soft/hard edge}
	\end{subfigure}	
    \caption{ The plots display graphs of $R^{(1)}(x)$ and $R^{(\infty)}(x)$ in \eqref{R free} and \eqref{R softhard}, where $\rho=4$.  The graphs are against the $x$-variable as they are invariant under translation in imaginary direction.  }
    \label{Fig_R free hard}
\end{figure}

In this note, we address this problem for almost-circular ensembles. 
To our knowledge, Theorems~\ref{thm:onelimit} and \ref{thm:hardwall two-sided} below provide the first concrete example of the scaling limits of the form \eqref{R softhard}. 
Therefore our results contribute to an affirmative answer for the conjecture in \cite{MR4030288}. 
Furthermore, beyond the free and the soft/hard edge boundary conditions, we shall consider the following two natural boundary conditions which interpolate between these cases:
\begin{itemize}
    \item \textbf{\emph{Soft edge with a confinement parameter}}; 
    \smallskip 
    \item \textbf{\emph{Two-sided hard edge cuts}}. 
\end{itemize}
The former one is a more potential theoretic boundary condition, whereas the latter one is a more geometric boundary condition.  

The soft edge boundary condition with a confinement parameter is defined by a linear interpolation between the external potential $Q_n$ and the logarithmic potential of the equilibrium measure.
We refer to \cite{MR4169375} for potential theoretic motivations on such boundary conditions. 
See also \cite{chafai2020macroscopic,garcia2018edge} for related boundary conditions.
This convex combination gives rise to universal point fields which interpolate between the point fields with the $1$-point function \eqref{R free} and with \eqref{R softhard}.

The two-sided hard edge boundary condition is obtained by moving the hard edge cuts that confine the gas to a narrow annulus between them. 
(We refer to \cite{seo2020edge,hedenmalm2020riemann} for similar situations.) 
This boundary condition provides another interpolation between the $1$-point functions \eqref{R free} and \eqref{R softhard}.
In particular, by moving the hard edge cuts away from the droplet, one can recover the free boundary limit \eqref{R free}.

In the following subsections, we precisely introduce such boundary conditions and formulate our results.

\subsection{Soft edge with a confinement parameter}
First, let us discuss the soft edge with a confinement parameter.
For this purpose, we consider the logarithmic potential 
\begin{equation} \label{U Qn}
U_{{Q_n}}(\zeta):=\int_\C \log \frac{1}{|\zeta-\eta|}\,d\sigma_n(\eta)
\end{equation}
of $\sigma_{Q_n}$. 
Recall that $\sigma_{Q_n}$ is the equilibrium measure of the energy functional \eqref{Ilog}.

The \textit{obstacle function} $\check{Q}_n$ pertaining to a potential $Q_n$ is defined by the maximal subharmonic function satisfying $\check{Q}_n \leq Q_n$ on $\C$ and $\check{Q}_n(\zeta) \le \log|\zeta|^2 + O(1)$ at infinity. 
Equivalently, it can be defined in terms of $U_{{Q_n}}$ in \eqref{U Qn} as 
\begin{equation} \label{Q check}
\check{Q}_n(\zeta)=C_n-2\,U_{Q_n}(\zeta),
\end{equation}
where $C_n$ is the modified Robin constant, a unique constant satisfying $2U_{Q_n}+Q_n =C_n$ on $S_n$ and $2U_{Q_n}+Q_n \ge C_n$ on $\C$ (q.e.), see \cite{ST}. 

For a radially symmetric potential \eqref{Q radially symmetric}, the logarithmic potential $U_{{Q_n}}$ can be computed explicitly. 
As a consequence, the obstacle function $\check{Q}_n$ in \eqref{Q check} is given by 
\begin{equation} \label{Q check radial}
\check{Q}_n(\zeta) = \begin{cases} 
Q_n(r_0) &  \text{if }  |\zeta| < r_0, 
\smallskip 
\\
Q_n(\zeta) & \text{if }  r_0 \leq |\zeta| \leq r_1,
\smallskip 
\\
2\log|\zeta| - 2\log r_1 + Q_n(r_1) & \text{otherwise},
\end{cases} 
\end{equation}
see e.g. \cite{MR4169375,ST}. (Recall that $r_0$ and $r_1$ are radii of the droplet $S_n$ in \eqref{S Qn}.) 
Using \eqref{Q check radial}, for a given parameter $\para = (\para_1,\para_2)$ with $\para_{1}, \para_2 \in \R_{+} \cup \{ \infty  \}$, we define 
\begin{equation} \label{Q interpolating}
Q_n^{(\para)}(\zeta) := \begin{cases}
\para_1 Q_n+(1-\para_1)\check{Q}_n(\zeta) &  \text{if }  |\zeta| < r_0, 
\smallskip 
\\
Q_n(\zeta) &  \text{if }  r_0 \leq |\zeta| \leq r_1, 
\smallskip 
\\
\para_2 Q_n+(1-\para_2)\check{Q}_n(\zeta) &  \text{otherwise}.
\end{cases}
\end{equation}
Note that the parameter $\para_1$ (resp., $\para_2$) determines the boundary condition at inner (resp., outer) circle of the droplet. 

To describe the scaling limits of almost-circular ensembles associated with the potential $Q_n^{(\para)}$, let us write
\begin{equation} \label{I para}
I^{(\para)}(t):= \frac{1-\para_1}{2}\Big(t+\frac{\rho}{2}\Big)_{-}^2 + \frac{1-\para_2}{2}\Big(t-\frac{\rho}{2}\Big)_{+}^2
\end{equation}
and
\begin{equation} \label{Phi para}
\Phi^{(\para)}(\xi) := \int_{-\infty}^{\infty} 
e^{-\frac{1}{2}(t-\xi)^2 + I^{(\para)}(t)}\, dt.
\end{equation}
(Here, we write $a_{+} = \max(a,0)$, $a_{-} = \min(a,0)$, and  $a_{+}^2=(a_{+})^2$, $a_{-}^2 = (a_{-})^2$ to avoid bulky notation.) 
These functions are building blocks to define
\begin{equation} \label{F para}
    F^{(\para)}(z) := \int_{-\frac{\rho}{2}}^{\frac{\rho}{2}} \frac{ e^{-\frac{1}{2}(z - \xi)^2}}{\Phi^{(\para)}(\xi)} \, d\xi.
\end{equation}

Let $\{\zeta_j\}_1^n$ be the particle system \eqref{Gibbs} in external potential $Q_n^{(\para)}$ and $\{z_j\}_1^n$ be the rescaled system at $\alpha_n=1$, see \eqref{z zeta}.
Notice that by \eqref{r0 r1}, the point $\alpha_n=1$ is asymptotically  the midpoint between $r_0$ and $r_1$. 
We write $R_n^{(\para)}$ for the $1$-point function of the rescaled system. 

We obtain the following result, which recovers Theorem~\ref{thm:onelimit_fH} with the specific choice $\para_1=\para_2=1$.

\begin{thm} \label{thm:onelimit} 
\textbf{\textup{(Scaling limits for the soft edge with a confinement parameter)}}
As $n \to \infty$, the $1$-point function $R_n^{(\para)}$ converges locally uniformly to the limit 
\begin{equation} \label{R int}
R^{(\para)}(z) = F^{(\para)}(2 \re z ) \cdot \exp(I^{(\para)}(2\re z)), 
\end{equation}
where $F^{(\para)}$ and $I^{(\para)}$ are given by \eqref{F para} and \eqref{I para}.  
\end{thm}

See Figure~\ref{Fig_R soft para} for the graphs of $R^{(\para)}$. 
We remark that the corresponding correlation kernel is given by
\begin{equation*}
    K^{(\para)}(z,w) = G(z,w)\, F^{(\para)}(z+\bar{w}) \cdot \exp\Big(\frac{1}{2}I^{(\para)}(2\re z) + \frac{1}{2}I^{(\para)}(2\re w)\Big),
\end{equation*}
see \cite[Theorem 1.1]{MR4169375} for the structure of correlation kernels.
It is easy to observe that in the extremal case $\para_1=\para_2=\infty$, the $1$-point function \eqref{R int} corresponds to $R^{(\infty)}$ in \eqref{R softhard} which appears in the soft/hard edge case. In this case, the limiting point field is strictly contained in the strip $\{z\in\C : |\re z|<\frac{\rho}{4}\}$.

\begin{figure}[h!]
    	\begin{subfigure}{0.32\textwidth}
		\begin{center}	
			\includegraphics[width=\textwidth]{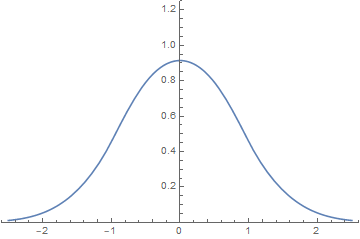}
		\end{center}
		\subcaption{$\para_1=\para_2=1/2$}
	\end{subfigure}	
	\begin{subfigure}{0.32\textwidth}
		\begin{center}	
			\includegraphics[width=\textwidth]{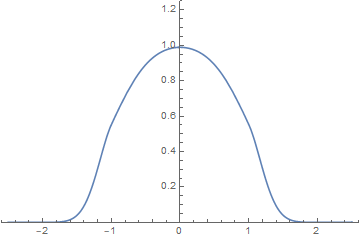}
		\end{center}
		\subcaption{$\para_1=\para_2=4$}
	\end{subfigure}	
	\begin{subfigure}{0.32\textwidth}
		\begin{center}	
			\includegraphics[width=\textwidth]{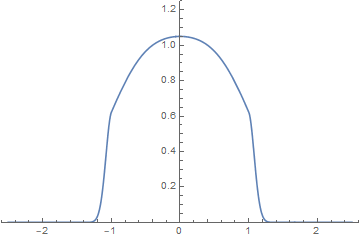}
		\end{center}
		\subcaption{$\para_1=\para_2=32$}
	\end{subfigure}	
    \caption{ The plots show graphs of $R^{(\para)}(x)$ in \eqref{R int} for a few values of $\para_1$ and $\para_2$, where $\rho=4$. One can observe that as $\para_1, \para_2 \to 1$, the graph of $R^{(\para)}$ tends to $R^{(1)}$ in Figure~\ref{Fig_R free hard} (A), whereas as $\para_1, \para_2 \to \infty$, it tends to $R^{(\infty)}$ in Figure~\ref{Fig_R free hard} (B). } \label{Fig_R soft para}
\end{figure}

\subsection{Hard edge confinements}
We now consider the \emph{hard edge cuts} near the droplet. 
For each $\tau \in \R$ with $0\leq \tau \leq 1$, there exists a unique number $r_\tau\equiv r_{n,\tau}$ such that 
\begin{equation} \label{r tau eq}
r_\tau \, g_n'(r_{\tau}) = 2\tau. 
\end{equation}
It is easy to observe from \eqref{r tau eq} that 
\begin{equation} \label{r tau asymp}
r_{\tau} = 1 - \frac{\rho^2}{4n}(1-2\tau) + o(n^{-1}),
\end{equation}
see Lemma~\ref{lem_r tau}. 
Let us write 
\begin{equation} \label{S tau Gamma tau}
S_\tau \equiv S_{n,\tau} := \{\zeta \in \C : r_0 \leq |\zeta| \leq r_\tau \}, \qquad   \Gamma_\tau\equiv \Gamma_{n,\tau} := \{\zeta \in \C : |\zeta| = r_{\tau}\}. 
\end{equation}
We now consider the case where the hard edge cuts are along the circles $\Gamma_{\tau_1}$ and $\Gamma_{\tau_2}$ so that the particles are confined in the thin annulus 
\begin{equation} \label{S tau2 tau1 hard}
S_{\tau_2}\setminus S_{\tau_1}= \{\zeta \in \C : r_{\tau_1} \leq |\zeta| \leq r_{\tau_2} \}, \qquad (0\leq \tau_1 < \tau_2 \leq 1).
\end{equation}
Equivalently, we consider the potential 
\begin{equation} \label{Q hard edge cuts}
Q_n^{\tau_1,\tau_2}(\zeta):=
\begin{cases}
Q_n(\zeta) &\text{if }  r_{\tau_1} \le |\zeta| \le r_{\tau_2},
\smallskip 
\\
+\infty &\text{otherwise}.
\end{cases}
\end{equation}
In the presence of hard edge cuts, the associated equilibrium measure $\sigma_{Q_n} $ is no longer absolutely continuous with respect to the area measure $dA$ in general, and takes the form 
\begin{equation} \label{eq msr hard edge}
d\sigma_{Q_n} = \Lap Q \cdot \1_{S_{\tau_2}\setminus S_{\tau_1}} \, dA + \frac{\tau_1}{r_{\tau_1}}\1_{\Gamma_{\tau_1}} \,ds + \frac{1-\tau_2}{r_{\tau_2}}\1_{\Gamma_{\tau_2}} \, ds.
\end{equation}
Here, $ds$ is the normalised arc length measure: $ds(z) = (2\pi)^{-1}|dz|$. 
One may realise from \eqref{eq msr hard edge} that there are non-trivial portions of particles near the boundaries of \eqref{S tau2 tau1 hard}.

More generally, for any real numbers $\tau_1, \tau_2$ ($\tau_1< \tau_2$), one can define $r_{\tau_1}$, $r_{\tau_2}$, and $Q_n^{\tau_1,\tau_2}$ when $n$ is sufficiently large. 
Note that when $\tau_1 < 0$ and $\tau_2>1$, the equilibrium measure associated with $Q^{\tau_1,\tau_2}$ is not affected by the hard edge cuts. 

We rescale the point process \eqref{Gibbs} associated with $Q_n^{\tau_1,\tau_2}$ in \eqref{Q hard edge cuts} at the zooming point $\alpha_n=1$, i.e. $z_j = \sqrt{n\Lap Q_n(1)}(\zeta_j - 1)$. 
We denote by $R_n^{\tau_1,\tau_2}$ the $1$-point function of the rescaled system. 
Then we obtain the following theorem. 

\begin{thm} \label{thm:hardwall two-sided} \textbf{\textup{(Scaling limits for the two-sided hard edge cuts)}}
As $n\to \infty$, the $1$-point function $R_n^{\tau_1,\tau_2}$ converges locally uniformly to the limit
\begin{equation} \label{R hardwall two-sided}
R^{\tau_1,\tau_2}(z) = \1_{\{\frac{\rho}{4}(2\tau_1-1)\leq x \leq \frac{\rho}{4}(2\tau_2-1)\}}\int_{-\rho/2}^{\rho/2} \frac{e^{-\frac{1}{2}(2x-\xi)^2}}{\int_{\frac{\rho}{2}(2\tau_1-1)}^{\frac{\rho}{2}(2\tau_2-1)} e^{-\frac{1}{2}(t-\xi)^2}\,dt} \, d\xi,\qquad x=\re z.
\end{equation}
\end{thm}

See Figure~\ref{Fig_R hardedge cuts} for the graphs of $R^{\tau_1,\tau_2}$ with a few values of $\tau_1$ and $\tau_2$. 

Note that the limiting point field of the rescaled system is completely confined in a strip which depends on the parameters $\tau_1$, $\tau_2$ of the hard edge cuts.
By \eqref{K structure}, it is easy to see that the correlation kernel associated with \eqref{R hardwall two-sided} is given by 
\begin{equation}
K^{\tau_1,\tau_2}(z,w) = G(z,w)\,\int_{-\rho/2}^{\rho/2} \frac{e^{-\frac{1}{2}(z+\bar{w}-\xi)^2}}{\int_{\frac{\rho}{2}(2\tau_1-1)}^{\frac{\rho}{2}(2\tau_2-1)} e^{-\frac{1}{2}(t-\xi)^2}\,dt} \, d\xi
\end{equation}
on the strip $\frac{\rho}{4}(2\tau_1-1) \leq \re z, \re w \leq \frac{\rho}{4}(2\tau_2-1)$, see e.g. \cite{MR4030288,AB21}.
We remark that in \cite{MR4071093}, the authors considered another notion of hard edge cuts, which yield different scaling limits. 

Notice that in the extremal case when $\tau_1 = -\infty$ and $\tau_2 = \infty$, the limiting $1$-point function $R^{\tau_1,\tau_2}$ corresponds to \eqref{R free}, which appears in the free boundary case.
On the other hand, if $\tau_1=0$ and $\tau_2=1$, it correspond to \eqref{R softhard}, which appears in the soft/hard edge case.

\begin{figure}[h!]
    	\begin{subfigure}{0.32\textwidth}
		\begin{center}	
			\includegraphics[width=\textwidth]{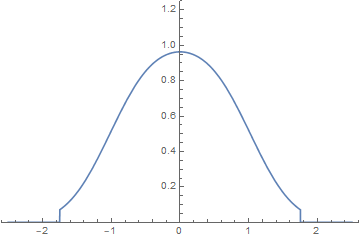}
		\end{center}
		\subcaption{$\tau_1=-\frac38$, $\tau_2=\frac{11}8$}
	\end{subfigure}	
	\begin{subfigure}{0.32\textwidth}
		\begin{center}	
			\includegraphics[width=\textwidth]{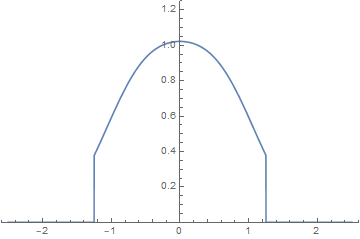}
		\end{center}
		\subcaption{$\tau_1=-\frac18$, $\tau_2=\frac98$}
	\end{subfigure}	
	\begin{subfigure}{0.32\textwidth}
		\begin{center}	
			\includegraphics[width=\textwidth]{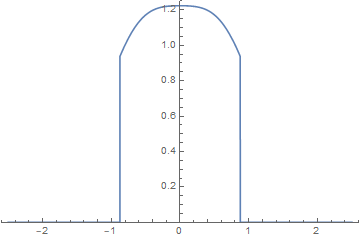}
		\end{center}
		\subcaption{$\tau_1=\frac{1}{16}$, $\tau_2=\frac{15}{16}$}
	\end{subfigure}	
    \caption{ The plots indicate graphs of $R^{\tau_1,\tau_2}(x)$ in \eqref{R hardwall two-sided}, where $\rho=4$. 
    One can easily observe that as $\tau_1 \to -\infty$ and $\tau_2 \to \infty$, the graph of $R^{\tau_1,\tau_2}$ tends to $R^{(1)}$ in Figure~\ref{Fig_R free hard} (A).
    On the other hand, as $\tau_1 \to 0$ and $\tau_2 \to 1$, it tends to $R^{(\infty)}$ in Figure~\ref{Fig_R free hard} (B). } \label{Fig_R hardedge cuts}
\end{figure}

\begin{rmk*} \textbf{(Circular limit $\rho \to 0$)} 
In the limit when $\rho \to 0$, one can recover the well-known sine kernel appearing in the circular unitary ensemble, see e.g. \cite{forrester2010log}.
To take such a limit, we rescale the kernel via $z\mapsto z/a$ with $a=\rho/2$. 
(This is due to the fact that we have kept the ``two-dimensional'' perspective, see e.g. \cite{AB21} for more details.)
The associated kernel $\widetilde{K}^{\tau_1,\tau_2}$ is given by 
\begin{equation}
\widetilde{K}^{\tau_1,\tau_2}(z,w) := \frac{1}{a^2} K^{\tau_1,\tau_2}(\frac{z}{a},\frac{w}{a}), \qquad  \frac{\rho^2}{8}(2\tau_1 -1)\leq \re z, \re w \leq \frac{\rho^2}{8}(2\tau_2-1). 
\end{equation}
Then it is easy to see that up to a cocycle, one can recover the sine kernel in the circular limit, i.e. 
\begin{equation} \label{Unitary limit}
\lim_{a \to 0} \widetilde{K}^{\tau_1,\tau_2}(z,w) = \frac{1}{\pi} \frac{\sin(u-v)}{u-v},  \qquad 
\begin{cases}
u=\im z,
\smallskip 
\\
v=\im w.
\end{cases}
\end{equation}
See Subsection~\ref{Subsec_limit hardwall} for further details.

\end{rmk*}

As a special case, let us consider the \emph{one-sided hard edge} case that the hard edge is along the outer circle $\Gamma_{\tau_2}$ but the inner one $\Gamma_{\tau_1}$ is a free boundary, i.e. $\tau_1=-\infty$ and $\tau_2=\tau$, cf. \eqref{r tau asymp} and \eqref{S tau Gamma tau}.  
This model corresponds to the external potential 
\begin{equation} \label{Q one sided}
    Q_{n}^{\tau}(\zeta) := 
    \begin{cases}
    Q_n(\zeta) &\text{if }  |\zeta| \le r_{\tau},
    \smallskip 
    \\
    +\infty &\text{otherwise}.
    \end{cases}
\end{equation}

We rescale the point process associated with the potential $Q_n^{\tau}$ in \eqref{Q one sided} at a point $r_\tau$ on the outer circle, i.e. $z_j = \sqrt{n\Lap Q_n(1)}(\zeta_j - r_\tau)$ and denote by $R_n^{\tau}$ the $1$-point function of the rescaled system. 

\begin{cor}  \label{cor:hardwall one-sided}
\textbf{\textup{(Scaling limits for the one-sided hard edge cuts)}}
As $n \to \infty$, the $1$-point function $R_n^{\tau}$ converges locally uniformly to the limit
\begin{equation} \label{R hardwall one-sided pre} 
R^\tau(z) = \1_{\{x\leq 0\}} \int_{-\rho\tau}^{\rho(1-\tau)} \frac{e^{-\frac{1}{2}(2x - \xi)^2}}{\int_{-\infty}^{0}e^{-\frac{1}{2}(t-\xi)^2 }\,dt}\, d\xi,\qquad x=\re z.
\end{equation}
In particular, the rescaled $1$-point function 
\begin{equation} \label{R hardwall one-sided}
\widetilde{R}^\tau(z):=\Big(\frac{c(\tau)}{\rho}\Big)^2 R^\tau\Big( \frac{c(\tau)}{\rho} z \Big),  \qquad c(\tau) := \Big(\sqrt{\frac{(1-\tau)^2}{4} + \frac{1}{\rho^2}} + \frac{1-\tau}{2}\Big)^{-1}
\end{equation}
satisfies 
\begin{equation} \label{R hardwall one-sided limit}
\lim_{\rho \to \infty} \widetilde{R}^\tau(z) =  
\begin{cases}
\displaystyle \1_{\{ x \leq 0\}}  \int_{-\infty}^{0} \frac{e^{-\frac{1}{2}(x-\xi)^2}}{\int_{-\infty}^{0} e^{-\frac{1}{2}(t-\xi)^2} \, dt} \, d\xi,  &\text{if }\tau=1,
\smallskip 
\\
\displaystyle  \1_{\{ x \leq 0\}} \int_{0}^{1}
\xi \cdot e^{2x\xi} \, d\xi, &\text{if }\tau \in (0,1). 
\end{cases}
\end{equation}
\end{cor}

See Figure~\ref{Fig_R onesided} for the graphs of $R^{\tau}(x)$.

We emphasise that in \eqref{R hardwall one-sided limit}, the former limit corresponds to the $1$-point function for the soft/hard edge Ginibre ensemble \cite{forrester2010log,MR3975882}, whereas the latter one corresponds to the $1$-point function for hard edge Ginibre ensemble \cite{seo2020edge}. 
Moreover, the latter limit also appears in the context of truncated unitary ensembles \cite{MR1748745,khoruzhenko2021truncations}. 

Let us briefly explain the rescaling factor $c(\tau)$ in \eqref{R hardwall one-sided}. 
The $1$-point function $\widetilde{R}^\tau$ appears when we rescale the point process as $z_j =\gamma_n(\zeta_j - r_\tau)$, where $\gamma_n$ satisfies
\begin{equation} \label{gmeqn}
\Lap Q(p) \, \gamma_n^2  + \frac{1-\tau}{r_{\tau}} \, \gamma_n =\frac{1}{n}.
\end{equation}
Note that by \eqref{eq msr hard edge}, the constant $\gamma_n$ corresponds to the mean eigenvalue spacing at $r_\tau$. 
Then the rescaling \eqref{R hardwall one-sided} follows from \eqref{gmeqn} since $\gamma_n \asymp c(\tau)/n$.

\begin{figure}[h!]
    	\begin{subfigure}{0.32\textwidth}
		\begin{center}	
			\includegraphics[width=\textwidth]{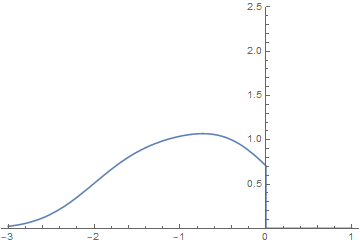}
		\end{center}
		\subcaption{$\tau=1$}
	\end{subfigure}	
	\begin{subfigure}{0.32\textwidth}
		\begin{center}	
			\includegraphics[width=\textwidth]{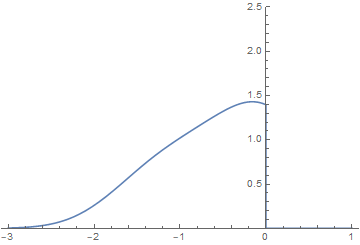}
		\end{center}
		\subcaption{$\tau=\frac56$}
	\end{subfigure}	
	\begin{subfigure}{0.32\textwidth}
		\begin{center}	
			\includegraphics[width=\textwidth]{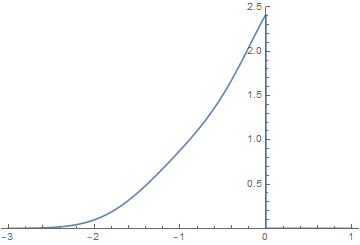}
		\end{center}
		\subcaption{$\tau=\frac{2}{3}$}
	\end{subfigure}	
    \caption{ The graphs of $R^{\tau}(x)$ in \eqref{R hardwall one-sided pre}. Here $\rho=4$. } \label{Fig_R onesided}
\end{figure}

\subsection{Maximal and minimal modulus}

In our final main result, we study the maximal and minimal modulus of almost-circular ensembles associated with the potential \eqref{Q interpolating}.

For the system $\{\zeta_j\}_{j=1}^{n}$ in \eqref{Gibbs} with the potential \eqref{Q interpolating}, we consider its maximal and minimal modulus
\begin{equation} \label{max min zeta}
|\zeta|_n := \max_{1\leq j\leq n} |\zeta_j|, \qquad  |\zeta|_1 := \min_{1\leq j\leq n} |\zeta_j|.
\end{equation}
This notation \eqref{max min zeta} will be used in the sequel. 
To investigate the fluctuation of the maximal (resp., minimal) modulus, we rescale $|\zeta|_n$ (resp., $|\zeta|_1$) near the outer (resp., inner) boundary of the droplet. 

Here, we should distinguish the soft/hard edge case.
Accordingly, we define random variables $\omega_n$ and $u_n$ as follows.

\begin{itemize}
    \item If $\para_1,\para_2 < \infty$, 
    \begin{equation} \label{maxmin soft with para}
\omega_n := a_n(|\zeta|_n - b_n), \qquad u_n := a_n'(b_n' - |\zeta|_1),
\end{equation}
where 
\begin{equation}
\begin{split} \label{an bn an' bn'}
a_n &= \frac{2n}{\rho}\sqrt{\para_2 \, c_n},\qquad b_n = r_1+ \frac{\rho}{2n}\sqrt{\frac{c_n}{\para_2}},
\\
a_n' &=  \frac{2n}{\rho}\sqrt{\para_1 \, c_n'},\qquad b_n' = r_0 -  \frac{\rho}{2n}\sqrt{\frac{c_n'}{\para_1}}. 
\end{split}
\end{equation}
Here 
\begin{equation}
\begin{split} \label{cn cn'}
c_n &= 2\log \frac{n}{\rho} - 2\log\log n - 2\log \big(2\, \Phi^{(\para)}(\frac{\rho}{2})\big),
\\
c_n' &= 2\log \frac{n}{\rho} - 2\log\log n - 2\log \big( 2\, \Phi^{(\para)}(-\frac{\rho}{2})\big). 
\end{split}
\end{equation}

\item If $\para_1=\para_2=\infty$, 
\begin{equation}  \label{maxmin soft/hard}
\omega_n := \frac{c\, n^2}{\rho^2}(|\zeta|_n - r_1),  \qquad  
u_n: = \frac{c'\, n^2}{\rho^2}(r_0 - |\zeta|_1),
\end{equation}
where 
\begin{equation} \label{c c'}
c = 2R^{(\infty)}(\frac{\rho}{4}), \qquad c' = 2R^{(\infty)}(-\frac{\rho}{4}).
\end{equation}
\end{itemize}

Here, $\Phi^{(\para)}$ and $R^{(\infty)}$ are given in \eqref{Phi para} and \eqref{R softhard} respectively.
We remark that the specific choice of parameters \eqref{an bn an' bn'}, \eqref{cn cn'}, and \eqref{c c'} leads to the universal form of the distributions \eqref{Gumbel}, \eqref{exp} below.
We obtain the following theorem.

\begin{thm}\label{Thm_M} \textup{\textbf{(Fluctuations of the maximal and minimal modulus)}}
The followings hold.
\begin{enumerate}[label=(\roman*)]
    \item If $\para_1,\para_2 < \infty$, the random variables $\omega_n$ and $u_n$ in \eqref{maxmin soft with para} converge in distribution to the standard Gumbel distribution, i.e.
    \begin{equation} \label{Gumbel}
\lim_{n\to \infty} \Prob_n(\omega_n \leq x) = \lim_{n\to \infty} \Prob_n(u_n \leq x) = e^{-e^{-x}}, \qquad x\in \R.
    \end{equation}
    \item If $\para_1=\para_2 = \infty$,  the random variables $\omega_n$ and $u_n$ in \eqref{maxmin soft/hard} converge in distribution to the standard exponential distribution, i.e. 
\begin{equation}\label{exp} 
\lim_{n\to \infty} \Prob_n(\omega_n \leq x) = \lim_{n\to \infty} \Prob_n(u_n \leq x) = e^{x}, \qquad x\in \R_{-}.
\end{equation}
\end{enumerate}
\end{thm}

This theorem asserts that the fluctuations of maximal and minimal modulus are same as those previously obtained for the usual random normal matrices \cite{chafai2014note,MR1148410, MR1986426, MR4179777} (and also planar symplectic ensembles \cite{MR1986426, Dubach2020}).
In other words, the distributions \eqref{Gumbel} and \eqref{exp} universally appear also in the context of almost-circular ensembles.

\subsection{Remarks on main theorems and glimpses of Ward's equations} \label{Subsec_Ward}

We end this section by giving some general remarks and briefly introducing alternative approach based on Ward's equations.  

The overall strategy for the proofs of our main theorems are as follows. 
Using the exact solvability of the model, we express the correlation kernel in terms of the orthonormal polynomials with respect to the weighted Lebesgue measure $e^{-nQ_n}\,dA,$ see \eqref{bfKn ONP}.
Since the external potential $Q_n$ is radially symmetric, the orthonormal polynomial is a monomial.
The asymptotics of the leading coefficient of such a polynomial can be obtained by virtue of Laplace methods, which allows us to derive large-$n$ limits of the correlation kernel by a proper Riemann sum approximation. 

Our proofs for the scaling limits (Theorems~\ref{thm:onelimit} and ~\ref{thm:hardwall two-sided}) are completely different with those used in \cite{AB21} for Theorem~\ref{thm:onelimit_fH}. 
In general, it can be shown that the limiting $1$-point function $R$ satisfies certain partial integro-differential equation called \emph{Ward's equation}. 
For the free boundary condition, it is of the form 
\begin{equation} \label{Ward free}
\pa C(z)=R(z)-1-\Lap \log R(z), \qquad C(z):=\frac{1}{R(z)}\int_\C \frac{|K(z,w)|^2}{z-w}\,dA(w).
\end{equation}
An important feature of this equation is that it \emph{does not} depend on the choice of the external potential.

In \cite{AB21}, the notion of ``cross-section convergence'' for general bandlimited point processes was introduced. 
Combining this property with the characterisation of translation invariant solutions to Ward's equation (obtained in \cite{MR4030288}), the authors derived various scaling limits of almost-Hermitian type ensembles under some natural geometric conditions notably the translation invariance of scaling limits.
In particular, since such conditions trivially hold for almost-circular ensembles, this immediately leads to Theorem~\ref{thm:onelimit_fH}.
(We refer the reader to \cite[Remark 2.9]{akemann2021scaling} for a summary of the approach using Ward's equation in the study of universality.) 

We briefly present Ward's equations in our various cases ($R^{(\para)}$ in \eqref{R int}; $R^{\tau_1,\tau_2}$ in \eqref{R hardwall two-sided}; $R^\tau$ in \eqref{R hardwall one-sided pre}; $\widetilde{R}^\tau$ in \eqref{R hardwall one-sided}) and compare them with those in previous literature.  
In the spirit of \cite{AB21,MR4030288,MR3975882}, this may be used to study the existence of further universality classes beyond radially symmetric ensembles.
We also refer to \cite{MR1487983} and \cite{akemann2021scaling} for implementations of Ward's equations in the study of Hermitian random matrices and planar symplectic ensembles respectively.
We remark that for a general class of external potentials, the local bulk and edge universality of random normal matrices were obtained in \cite{MR2817648} and \cite{hedenmalm2017planar} respectively. 
(See also a recent work \cite{ameur2021szego} which establishes Szeg\H o type asysmptotics for the correlation kernel.)

The derivation of Ward's equations follows from the standard method found for instance in \cite{MR3975882,MR3342661} and we skip the proof. 
The main idea is the reparametrisation invariance of the partition function $Z_n$.

In each situation, we write $C^{(\para)}$, $C^{\tau_1,\tau_2}$, $C^\tau$ and $\widetilde{C}^\tau$ for the associated Cauchy transforms of the form \eqref{Ward free}.

\begin{itemize}
    \item \textit{Ward's equation for soft edge with a confinement parameter.} The $1$-point function $R^{(\para)}$ in \eqref{R int} satisfies 
\begin{equation} \label{Ward para}
\bp C^{(\para)}(z)=R^{(\para)}(z)-1-\Lap \log R^{(\para)}(z)+(1-\para_1) \1_{\re z<-\frac{\rho}{4}}+(1-\para_2) \1_{\re z>\frac{\rho}{4}}.
\end{equation}
We remark that if $\para_1=1$ or $\para_2=1$, the equation \eqref{Ward para} corresponds to the one presented in \cite[Theorem 1.1]{MR4169375}.
Furthermore, if $\para_1=\para_2=1$, then \eqref{Ward para} corresponds to \eqref{Ward free}. 
\smallskip 
   \item \textit{Ward's equation for two-sided hard edge.} The $1$-point function $R^{\tau_1,\tau_2}$ in \eqref{R hardwall two-sided} satisfies
\begin{equation} \label{Ward two sided hard}
\bp C^{\tau_1,\tau_2}(z)=R^{\tau_1,\tau_2}(z)-1-\Lap \log R^{\tau_1,\tau_2}(z), \qquad \frac{\rho}{4}(2\tau_1-1) \le \re z \le \frac{\rho}{4}(2\tau_2-1).
\end{equation}
Such form of Ward's equation was extensively studied in \cite{MR4030288}.
Notice that if $\tau_1 \to -\infty$ and $\tau_2 \to \infty$, then \eqref{Ward two sided hard} corresponds to \eqref{Ward free}. 
   \smallskip 
   \item \textit{Ward's equation for one-sided hard edge.} The $1$-point function $R^\tau$ in \eqref{R hardwall one-sided pre} satisfies 
\begin{equation} \label{Ward one sided hard}
\bp C^\tau(z)=R^\tau(z)-1-\Lap \log R^\tau(z), \qquad \re z \le 0.
\end{equation}
The equation \eqref{Ward one sided hard} corresponds to the one presented in \cite[Subsection 7.1]{MR3975882}. 
By the change of variable, it is easy to see that the $1$-point function $\widetilde{R}^\tau$ in \eqref{R hardwall one-sided} satisfies 
\begin{equation} \label{Ward one sided hard res}
\bp \widetilde{C}^\tau(z)=\widetilde{R}^\tau(z)-( \frac{ c(\tau) }{ \rho} )^2-\Lap \log \widetilde{R}^\tau(z), \qquad \re z \le 0.
\end{equation}
Note that if $0<\tau<1$, we have 
$c(\tau)/\rho \to 0$ as $\rho \to \infty$. 
Therefore by taking the limit $\rho \to \infty$ of \eqref{Ward one sided hard res}, we arrive at 
\begin{equation} \label{Ward hardW}
\bp \widetilde{C}^\tau(z)=\widetilde{R}^\tau(z)-\Lap \log \widetilde{R}^\tau(z), \qquad \re z \le 0, \qquad (\rho=\infty). 
\end{equation}
This form \eqref{Ward hardW} of Ward's equation was used in \cite{seo2020edge} to study scaling limits of hard edge ensembles. 
\end{itemize}

 The rest of this note is structured as follows.
In Section~\ref{Section_Scaling limits}, we derive various scaling limits (Theorems~\ref{thm:onelimit}, ~\ref{thm:hardwall two-sided} and Corollary~\ref{cor:hardwall one-sided}).
Section~\ref{Section_modulus} is devoted to the study of maximal and minimal modulus of the ensemble, Theorem~\ref{Thm_M}.

\section{Scaling limits of almost-circular ensembles}\label{Section_Scaling limits}

In this section, we show Theorems~\ref{thm:onelimit}, ~\ref{thm:hardwall two-sided} and Corollary~\ref{cor:hardwall one-sided}. 

\subsection{Soft edge with a confinement parameter}

In this subsection, we prove Theorem~\ref{thm:onelimit}.

Recall that for each real $\tau$, $r_\tau=r_{n,\tau}$ denotes a unique constant satisfying \eqref{r tau eq}. 
The following lemma gives the asymptotic expansion of $r_{\tau}$.

\begin{lem}  \label{lem_r tau}
For each fixed $\tau$, we have
\begin{equation}
r_{\tau} = 1 - \frac{\rho^2}{4n}(1-2\tau) + o(n^{-1}),\qquad n\to \infty.
\end{equation}
\end{lem}

\begin{proof} 
Since $g_n(1)=0$ and $g_n'(1) = 1$, we have $r_{\tau}=1$ when $\tau=\frac{1}{2}$. By Taylor series expansion, it follows from \eqref{rho} that 
\begin{align*}
    1-2\tau &= g_n'(1) - r_{\tau} g_n'(r_{\tau}) = \frac{d}{dr}(rg'(r))|_{r=1} \cdot (1 -r_{\tau}) \cdot (1+o(1)) \\
    &=4\Lap Q_n(1)\cdot (1-r_{\tau}) \cdot (1+o(1)) = \frac{4 n}{\rho^2}\cdot (1-r_{\tau})\cdot (1+o(1)),
\end{align*}
which completes the proof.
\end{proof}

Let us write 
\begin{equation} \label{delta n}
\delta_n = n^{-1}\log n .
\end{equation}
By \eqref{Q check radial}, it is easy to observe that for $r_0 - \delta_n < r <r_0$,
\begin{equation}\label{exp:atr0}
(Q_n - \check{Q}_n)(r) = 2\Lap Q_n(r_0)(r-r_0)^2 + O(n|r-r_0|^3) = \frac{2n}{\rho^2}(r-r_0)^2 \cdot (1+o(1)) + O(n\delta_n^3)
\end{equation}
as $n\to \infty$.
Similarly, for $r_1<r < r_1 + \delta_n$,
\begin{equation}\label{exp:atr1}
(Q_n - \check{Q}_n)(r) = 2\Lap Q_n(r_1)(r-r_1)^2 + O(n|r-r_1|^3) = \frac{2n}{\rho^2}(r-r_1)^2 \cdot (1+o(1)) + O(n\delta_n^3).
\end{equation}

We first prove some estimates for the $L^2$-norm of monomials with respect to the measure $e^{-nQ_n^{(\para)}}\,dA$. 

For each $j$ with $0\leq j\leq n-1$, we define the functions $v_{n,j}$ by
\begin{equation}\label{v_nj}
v_{n,j}(r) = g_n(r) - 2\tau(j) \log r, \qquad \tau(j) := j/n.
\end{equation}
The function $v_{n,j}$ has a unique critical point at $r_{\tau(j)}$. We obtain the expansion 
\begin{align} \label{exp_v}
\begin{split}
v_{n,j}(r) 
&= v_{n,j}(r_{\tau(j)}) + 2\Lap Q(r_{\tau(j)})(r-r_{\tau(j)})^2 + O(n\delta_n^3)\\
&= v_{n,j}(r_{\tau(j)}) + \frac{2n}{\rho^2}(r-r_{\tau(j)})^2\cdot(1+o(1)) + O(n\delta_n^3)
\end{split}
\end{align}
for all $r$ with $|r-r_{\tau(j)}|\leq \delta_n$. 
Recall that $\Phi^{(\para)}$ is given by \eqref{Phi para}.

We write 
\begin{equation} \label{orthogonal norm para}
\|\zeta^j\|_{nQ_n^{(\para)}}^2 
:= \int_{\C} |\zeta^j|^2 e^{-nQ_n^{(\para)}(\zeta)} \, dA (\zeta)
\end{equation}
for the orthogonal norm. 
Then we obtain the following asymptotic behaviour of \eqref{orthogonal norm para}.

\begin{lem} \label{lem:L2norm}
For each $j$ with $0\leq j \leq n$, we have 
\begin{equation}\label{jnorm}
\|\zeta^j\|_{nQ_n^{(\para)}}^2 = e^{-nv_{n,j}(r_{\tau(j)})}\frac{\rho}{n}\, \Phi^{(\para)}\Big(\rho(\tau(j) - \frac{1}{2})\Big) \cdot (1+o(1)),\end{equation}
where $\tau(j) = j/n$ and $o(1)\to 0$ uniformly in $j$.
\end{lem}
\begin{proof}
In order to evaluate \eqref{orthogonal norm para}, we split the integral into the sum of three integrals:  
\begin{align*}
\|\zeta^j\|_{nQ_n^{(\para)}}^2 
& = \int_{S_n} |\zeta^j|^2 e^{-nQ_n(\zeta)} \, dA(\zeta) 
\\
&\quad + \int_{|\zeta|<r_{0}} |\zeta^j|^2 e^{-nQ_n^{(\para)}(\zeta)}\, dA(\zeta) + \int_{|\zeta|>r_{1}} |\zeta^j|^2 e^{-nQ_n^{(\para)}(\zeta)} \, dA(\zeta),
\end{align*}
where $S_n$ is given by \eqref{S Qn}. 
By \eqref{exp_v}, the standard Laplace's method gives rise to 
\begin{align*}
\int_{S_n}|\zeta^j|^2 e^{-nQ_n(\zeta)}\, dA(\zeta) 
&= \int_{r_{0}}^{r_{1}} e^{-nv_{n,j}(r)}\,2 r\, dr 
\\
&= e^{-nv_{n,j}(r_{\tau(j)})} \!\!\int_{r_{0}}^{r_{1}} e^{-\frac{2n^2}{\rho^2}(r-r_{\tau(j)})^2} \,2\, dr \cdot (1+o(1))\\
&= e^{-nv_{n,j}(r_{\tau(j)})} \!\! \int_{-\frac{\rho}{2}}^{\frac{\rho}{2}} e^{-\frac{1}{2} (t-\rho(\tau(j) - \frac{1}{2}))^2} \frac{\rho}{n} \, dt \cdot (1+o(1)),
\end{align*}
where we use Lemma \ref{lem_r tau} and make the change of variable $t = \frac{2n}{\rho}(r-1)$. 
In the same way, we obtain from the asymptotic expansions \eqref{exp:atr0}, \eqref{exp:atr1}, and \eqref{exp_v} that  
\begin{align*}
\int_{|\zeta|<r_{0}} \!\!\!|\zeta^j|^2 e^{-nQ_n^{(\para)}(\zeta)} \, dA(\zeta) 
& = \int_{0}^{r_{0}} e^{-n(v_{n,j}(r)-(1-\para_1)(Q_n-\check{Q}_n)(r))} \,2r\, dr \\
& = e^{-nv_{n,j}(r_{\tau(j)})} \!\!\int_{r_0 -\delta_n}^{r_{0}} e^{-\frac{2n^2}{\rho^2} ((r-r_{\tau(j)})^2 - (1-\para_1)(r-r_{0})^2) } \, 2 \, dr \cdot (1+o(1))\\
& = e^{-nv_{n,j}(r_{\tau(j)})} \!\!\int_{-\infty}^{-\frac{\rho}{2}}\!e^{-\frac{1}{2} ((t-\rho(\tau(j)-\frac{1}{2}))^2 - (1-\para_1)(t+\frac{\rho}{2})^2) } \frac{\rho}{n} \, dt \cdot (1+o(1))
\end{align*}
and
\begin{align*}
\int_{|\zeta|>r_{1}} \!\!\!|\zeta^j|^2 e^{-nQ_n^{(\para)}(\zeta)}\, dA(\zeta) 
& = \int_{r_{1}}^{\infty} e^{-n(v_{n,j}(r)-(1-\para_2)(Q_n-\check{Q}_n)(r))} \,2r \,  dr \\
& = e^{-nv_{n,j}(r_{\tau(j)})} \!\!\int_{r_{1}}^{r_1+\delta_n} e^{-\frac{2n^2}{\rho^2} ((r-r_{\tau(j)})^2 - (1-\para_2)(r-r_{1})^2)  } \, 2 \, dr \cdot (1+o(1))\\
& = e^{-nv_{n,j}(r_{\tau(j)})} \!\!\int_{\frac{\rho}{2}}^{\infty}\!e^{-\frac{1}{2} ((t-\rho(\tau(j)-\frac{1}{2}))^2 - (1-\para_2)(t-\frac{\rho}{2})^2)  } \frac{\rho}{n} \, dt \cdot (1+o(1)).
\end{align*}
Combining all of the above, we conclude \eqref{jnorm}.
\end{proof}

We now prove Theorem~\ref{thm:onelimit}.

\begin{proof}[Proof of Theorem~\ref{thm:onelimit}]
By \eqref{bfKn det}, \eqref{bfKn ONP}, and \eqref{R_n}, the $1$-point function $R_n^{(\para)}$ rescaled at $\point_n=1$ is written as
\begin{equation}\label{Rn para}
R_n^{(\para)}(z) = \frac{1}{n\Lap Q_n(1)} \sum_{j=0}^{n-1} \frac{|\zeta^j|^2}{\|\zeta^j\|_{nQ_n^{(\para)}}^2} e^{-nQ_n^{(\para)}(\zeta)}, \qquad \zeta = 1 + \frac{z}{\sqrt{n\Lap Q(1)}}.
\end{equation}
We shall analyse the asymptotic behaviour of \eqref{Rn para}.

Fix a compact subset $\calK $ in $\C$. For $\zeta = 1 + z/\sqrt{n\Lap Q_n(1)}$ with $r_{0} \leq |\zeta|\leq r_{1}$ and $z \in \calK$, it follows from the expansion \eqref{exp_v} that  
\begin{align*}
Q_n^{(\para)}(\zeta) - 2\tau(j) \log | \zeta| 
&= v_{n,j}(r_{\tau(j)}) + \frac{2n}{\rho^2} \Big(1 + \frac{x}{ \sqrt{n\Lap Q_n(1)} } - r_{\tau(j)}\Big)^2 \cdot (1+o(1)) 
\\
&= v_{n,j}(r_{\tau(j)}) + \frac{1}{2n} \Big( 2x - \rho(\tau(j) - \frac{1}{2}) \Big)^2 \cdot (1+o(1)),
\end{align*}
where $\tau(j)=j/n$ and $x = \re z$. 
Similarly, for $|\zeta| < r_{0}$ and $z \in \calK$, we have 
\begin{align*} 
Q_n^{(\para)}(\zeta) - 2\tau(j) \log | \zeta|  
&= v_{n,j}(\zeta) - (1-\para_1)(Q-\check{Q})(\zeta)
\\
&= v_{n,j}(r_{\tau(j)}) + \frac{1}{2n}\Big( (2x - \rho(\tau(j) - \frac{1}{2}))^2 - (1-\para_1)(2x+ \frac{\rho}{2})^2 \Big) \cdot (1+o(1)).
\end{align*}
Finally, for $|\zeta|>r_{1}$ and $z\in \calK$, we have
\begin{align*} 
Q_n^{(\para)}(\zeta) - 2\tau(j) \log | \zeta| 
&= v_{n,j}(\zeta) - (1-\para_2)(Q-\check{Q})(\zeta)
\\
&= v_{n,j}(r_{\tau(j)}) + \frac{1}{2n} \Big( (2x - \rho(\tau(j) - \frac{1}{2}))^2 - (1-\para_2)(2x - \frac{\rho}{2})^2 \Big) \cdot (1+ o(1)).
\end{align*}
Here, the $o(1)$-terms are bounded uniformly for all $j$ with $0\leq j\leq n-1$ and all $z \in \calK$. 

Combining all of the above with Lemma \ref{lem:L2norm}, we obtain the following Riemann sum approximation of \eqref{Rn para}: 
\begin{align*}
R_n^{(\para)}(z) &= \sum_{j=0}^{n-1} \frac{\rho \, e^{-\frac{1}{2}(2x - \rho(\tau(j) -\frac{1}{2}))^2 }}{n\Phi^{(\para)}(\rho(\tau(j)-\frac{1}{2}))} e^{\frac{1-\para_1}{2}(2x + \frac{\rho}{2})^2\cdot \1_{\{x < -\frac{\rho}{4}\}} + \frac{1-\para_2}{2}(2x - \frac{\rho}{2})^2 \cdot \1_{\{x > \frac{\rho}{4}\}}} \cdot (1+o(1))\\
& = \int_{-\frac{\rho}{2}}^{\frac{\rho}{2}} \frac{e^{-\frac{1}{2}(2x - \xi)^2}}{\Phi^{(\para)}(\xi)} d\xi \cdot 
e^{\frac{1-\para_1}{2}(2x + \frac{\rho}{2})^2\cdot \1_{\{x < -\frac{\rho}{4}\}} + \frac{1-\para_2}{2}(2x - \frac{\rho}{2})^2 \cdot \1_{\{x > \frac{\rho}{4}\}}} \cdot (1+o(1)),
\end{align*}
where $o(1)\to 0$ uniformly for $z \in \calK$ as $n\to \infty$.
This completes the proof. 
\end{proof}

\subsection{Hard edge case} \label{Subsec_limit hardwall}

In this subsection, we prove Theorem~\ref{thm:hardwall two-sided}.
The strategy is similar to the one in the previous subsection.

\begin{proof}[Proof of Theorem~\ref{thm:hardwall two-sided}]
Again we need to analyse the orthogonal norm
\begin{equation} \label{orthogonal norm hard edge cuts}
\|\zeta^j\|^2_{n Q_n^{\tau_1,\tau_2}} := \int_{S_{\tau_2}\setminus S_{\tau_1}} |\zeta|^{2j} e^{-nQ_n(\zeta)} \, dA(\zeta) = \int_{r_{\tau_1}}^{r_{\tau_2}} e^{-nv_{n,j}(r)} \,2r\, dr 
\end{equation}
associated with the potential \eqref{Q hard edge cuts}.
Here $v_{n,j}(r) = g_n(r) - \frac{2j}{n}\log r$, cf. \eqref{Q radially symmetric}.  
By the Taylor series expansion, with $\tau(j)=j/n$, we have
$$
v_{n,j}(r) = v_{n,j}(r_{\tau(j)}) + \frac{2n}{\rho^2}(r-r_{\tau(j)})^2 \cdot (1+o(1)) + O(n|r-r_{\tau(j)}|^3).
$$
As in Lemma \ref{jnorm}, the $L^2$-norm \eqref{orthogonal norm hard edge cuts} is approximated by
\begin{align} \label{orthogonal norm hard edge cuts asym}
\begin{split}
\|\zeta^j\|^2_{nQ_n^{\tau_1,\tau_2}} 
&= e^{-nv_{n,j}(r_{\tau(j)})}\int_{r_{\tau_1}}^{r_{\tau_2}} e^{-\frac{2n^2}{\rho^2}(r-r_{\tau(j)})^2} \,2\,dr \cdot(1+o(1))\\
&= e^{-nv_{n,j}(r_{\tau(j)})}\frac{\rho}{n} \int_{\frac{\rho}{2}(2\tau_1-1)}^{\frac{\rho}{2}(2\tau_2-1)} e^{-\frac{1}{2}(t-\frac{\rho}{2}(2\tau(j)-1))^2} \, dt \cdot(1+o(1)).
\end{split}
\end{align}

Note that by \eqref{bfKn det}, \eqref{bfKn ONP}, and \eqref{R_n}, the $1$-point function $R_n^{\tau_1,\tau_2}$ associated with the hard edge potential \eqref{Q hard edge cuts} is written as 
\begin{equation} \label{Rn hard edge cuts}
R_n^{\tau_1,\tau_2}(z)  = \frac{1}{n\Lap Q(1)}\sum_{j=0}^{n-1}\frac{|\zeta^j|^2}{\|\zeta^j\|^2_{nQ_n^{\tau_1,\tau_2}}} e^{-nQ_n(\zeta)} \1_{S_{\tau_2}\setminus S_{\tau_1}}(\zeta) .
\end{equation}
Then by \eqref{orthogonal norm hard edge cuts asym}, the right-hand side of \eqref{Rn hard edge cuts} is approximated as 
\begin{align*}
R_n^{\tau_1,\tau_2}(z) 
& = \frac{\rho}{n}\1_{\{\frac{\rho}{4}(2\tau_1-1)\leq x \leq \frac{\rho}{4}(2\tau_2-1)\}}
\sum_{j=0}^{n-1} \frac{e^{-\frac{2n^2}{\rho^2}(1+\frac{\rho}{n}x-r_{\tau(j)})^2}}{\int_{\frac{\rho}{2}(2\tau_1-1)}^{\frac{\rho}{2}(2\tau_2-1)} e^{-\frac{1}{2}(t - \frac{\rho}{2}(2\tau(j)-1))^2} \, dt  } \cdot(1+o(1))
\\
&= \frac{\rho}{n}\1_{\{\frac{\rho}{4}(2\tau_1-1)\leq x \leq \frac{\rho}{4}(2\tau_2-1)\}}
\sum_{j=0}^{n-1} \frac{e^{-\frac{1}{2}(2x- \frac{\rho}{2}(2\tau(j)-1))^2}}{\int_{\frac{\rho}{2}(2\tau_1-1)}^{\frac{\rho}{2}(2\tau_2-1)} e^{-\frac{1}{2}(t - \frac{\rho}{2}(2\tau(j)-1))^2} \, dt } \cdot(1+o(1)).
\end{align*}
Here, $z = \sqrt{n\Lap Q(1)}(\zeta-1)$ and $x = \re z$.
Now theorem follows from the Riemann sum approximation.
\end{proof}

Before moving on to the proof of Corollary~\ref{cor:hardwall one-sided}, let us briefly discuss the circular limit \eqref{Unitary limit}. 
Note that for $z,w$ with $ \frac{\rho^2}{8}(2\tau_1 -1) \leq \re z,\re w \leq \frac{\rho^2}{8}(2\tau_2-1) $
\begin{equation*}
 \widetilde{K}^{\tau_1,\tau_2}(z,w) := \frac{1}{a^2} K^{\tau_1,\tau_2}(\frac{z}{a},\frac{w}{a})
    =\frac{1}{a^2}\, e^{\frac{1}{a^2}(z\bar{w}-\frac{1}{2}|z|^2 - \frac{1}{2}|w|^2)} 
    \int_{-\rho/2}^{\rho/2} \frac{e^{-\frac{1}{2}(\frac{1}{a}(z+\bar{w})-\xi)^2}}{\int_{\rho(\tau_1-\frac{1}{2})}^{\rho(\tau_2-\frac{1}{2})} e^{-\frac{1}{2}(t-\xi)^2}\, dt } \, d\xi,
\end{equation*}
where $a=\rho/2$. As $\rho \to 0$, (i.e. $a \to 0$), we have
\begin{align*}
    \widetilde{K}^{\tau_1,\tau_2}(z,w) 
    &= \frac{1}{a^2} e^{-\frac{1}{a^2}((\re z)^2 + (\re w)^2)}e^{-\frac{1}{a^2}i \im (z^2 - w^2)} \int_{-\rho/2}^{\rho/2} \frac{e^{\frac{1}{a}\xi(z+\bar{w}) -\frac{1}{2}\xi^2 } }{ \int_{\rho(\tau_1-\frac{1}{2})}^{\rho(\tau_2-\frac{1}{2})} e^{-\frac{1}{2}(t-\xi)^2}\, dt }\, d\xi \cdot (1+o(1))
    \\
    &= \frac{1}{2a^2(\tau_2-\tau_1)} e^{-\frac{1}{a^2}((\re z)^2 + (\re w)^2)}c(z,w) \int_{-1}^{1} e^{\xi(z+\bar{w})} \, d\xi \cdot (1+o(1)).
\end{align*}
Here $c(z,w) = \exp(-\frac{1}{a^2}i \im(z^2-w^2))$ is a cocycle, which does not contribute when forming the determinant, see e.g. \cite{MR3975882}.  
Notice that for any test function $f$,
\begin{equation*}
    \int_{a^2(\tau_1-\frac{1}{2})}^{a^2(\tau_2-\frac{1}{2})} \frac{1}{a^2(\tau_2-\tau_1)} e^{-\frac{1}{a^2} x^2} f(x) \, dx \to f(0),\qquad a\to 0,
\end{equation*}
which leads to \eqref{Unitary limit}.

We now prove Corollary~\ref{cor:hardwall one-sided}.

\begin{proof}[Proof of Corollary~\ref{cor:hardwall one-sided}]

The first assertion is an immediate consequence of Theorem~\ref{thm:hardwall two-sided}.

We shall show the second assertion \eqref{R hardwall one-sided limit}. 
Note that the first limit when $\tau=1$ is trivial.
In  the case when $0<\tau <1$, it follows from \eqref{R hardwall one-sided} that 
\begin{equation} \label{RHW}
\widetilde{R}^\tau(z) = \1_{\{-\infty < x \leq 0\}} \frac{c}{\rho} \int_{-c\tau}^{c(1-\tau)}   
\frac{e^{-\frac{1}{2}(\frac{2c}{\rho}x - \frac{\rho}{c}\xi)^2}}{\int_{-\infty}^{0}e^{-\frac{1}{2}(t-\frac{\rho}{c}\xi)^2 } \, dt } \, d\xi, \qquad c=c(\tau).
\end{equation}
Observe that for each $\xi \in [-c\tau , 0]$, we have
\begin{align*}
\lim_{\rho\to \infty}
\frac{c}{\rho}   
\frac{e^{-\frac{1}{2}(\frac{2c}{\rho}x - \frac{\rho}{c}\xi)^2}}{\int_{-\infty}^{0}e^{-\frac{1}{2}(t-\frac{\rho}{c}\xi)^2 } dt} =0.
\end{align*}
On the other hand, for each $\xi\in [0, c(1-\tau)]$, it follows from the asymptotic expansion
$$ 
\int_{-\infty}^{0} e^{-\frac{1}{2}(t-\frac{\rho}{c}\xi)^2} \, dt =  e^{-\frac{1}{2}(\frac{\rho}{c}\xi)^2} \Big(\frac{c}{\xi\rho} + O\Big(\frac{c}{\xi\rho}\Big)^3\Big),\qquad \rho \to \infty,$$
that
\begin{align*}
\lim_{\rho\to\infty}\frac{c}{\rho}   
\frac{e^{-\frac{1}{2}(\frac{2c}{\rho}x - \frac{\rho}{c}\xi)^2}}{\int_{-\infty}^{0}e^{-\frac{1}{2}(t-\frac{\rho}{c}\xi)^2 } dt} = \xi \cdot e^{2x\cdot\xi}.
\end{align*}
This completes the proof. 

\end{proof}

\section{Maximal and minimal modulus} \label{Section_modulus}

In this section, we prove Theorem \ref{Thm_M}. 

Note that by \eqref{Gibbs}, the distribution of the maximal modulus $|\zeta|_n =\max_{1\leq j \leq n} |\zeta_j|$ is written as 
\begin{equation}  \label{Pn max}
	\Prob_n(|\zeta|_n \leq r) = \frac{1}{Z_n}\int_{D(0,r)^n} e^{-\Ham_n } \prod_{j=1}^{n} dA(\zeta_j).
\end{equation}
We shall analyse the asymptotic behaviour of \eqref{Pn max}.

\subsection{Soft edge with a confinement parameter} 
Since the external potential is radially symmetric, the distribution function \eqref{Pn max} of $|\zeta|_n$ has a closed form 
\begin{equation}\label{gap_prob}
	\Prob_n(|\zeta|_n \leq r) = \prod_{j=0}^{n-1} \Big(1-\int_{\C\setminus D(0,r)} \frac{|\zeta|^{2j} e^{-nQ^{(\para)}_n(\zeta)}}{\|\zeta^j\|^2_{nQ^{(\para)}_n}}\,dA(\zeta)\Big),
\end{equation}
see e.g. \cite{MR1148410, MR1986426}. Recall that the rescaled modulus $\omega_n$ is defined by \eqref{maxmin soft with para}. 

Write 
\begin{equation}\label{epsilon n}
\epsilon_n(x) = \frac{\rho}{2n\sqrt{\para_2}}\Big(\sqrt{c_n} + \frac{x}{\sqrt{c_n}}\Big), \qquad x\in \R,
\end{equation}
where $c_n$ is given in \eqref{cn cn'}.
Observe that 
\begin{align*}
	\Prob_n(\omega_n \leq x) 
	&= \Prob_n(|\zeta|_n \leq r_1+ \epsilon_n(x)) =\prod_{j=0}^{n-1}\Big(1-  \int_{r_1+\epsilon_n(x)}^{\infty}\frac{2re^{-n(v_{n,j}(r)-(1-\para_2)(Q_n - \check{Q}_n)(r)) }}{\|\zeta^j\|^2_{nQ^{(\para)}_n}}\,dr\Big),
	\end{align*}
where $v_{n,j}$ is defined in \eqref{v_nj}. 

Let us write
\begin{equation} \label{Sn sum}
	E_n(x) : = \sum_{j=0}^{n-1}\int_{r_1+\epsilon_n(x)}^{\infty}  \frac{2re^{-n(v_{n,j}(r)-(1-\para_2)(Q_n-\check{Q}_n)(r))}}{\|\zeta^j\|^2_{nQ_n^{(\para)}}}\, dr.
\end{equation}
The asymptotic of \eqref{Sn sum} is given as follows. 

\begin{lem}\label{S_N}
The function \eqref{Sn sum} satisfies 
\begin{equation} \label{Sn sum lim}
\lim_{n\to\infty} E_n(x) = e^{-x},
\end{equation}
where the convergence is uniform on every compact subset of $\R$.
\end{lem}
\begin{proof}
By Lemma \ref{lem_r tau}, the function $v_{n,j}$ has a critical point at $r_{\tau(j)}$, where $\tau(j)=j/n$ and 
$r_{\tau}$ satisfies the asymptotic expansion
$r_{\tau} = 1 - \frac{\rho^2}{4n}(1-2\tau) + o(n^{-1})$.
By Lemma \ref{lem:L2norm}, we obtain 
\begin{align*}
	&\quad {\|\zeta^{j}\|_{nQ^{(\para)}_n}^{-2}}\int_{r_1+\epsilon_n(x)}^{\infty}\!\!\! {2r e^{-n(v_{n,j}(r) -(1-\para_2)(Q_n-\check{Q}_n)(r)) }}\, dr 
	\\
	&= \frac{2n}{\rho\, \Phi^{(\para)}(\rho(\tau(j)-\frac{1}{2}))}\int_{r_1+\epsilon_n(x)}^{\infty}\!\! re^{-n(v_{n,j}(r)-v_{n,j}(r_{\tau(j)}) - (1-\para_2)(Q_n-\check{Q}_n)(r))}\,dr\cdot(1+o(1)),
\end{align*} 
where $o(1)\to 0$ as $n\to \infty$ uniformly for $j$ with $0\leq j \leq n-1$.

Recall that $\delta_n$ is given by \eqref{delta n}. By Taylor series expansion, we have
\begin{align}
	\begin{split}
		&\quad \int_{r_1+\epsilon_n(x)}^{r_1+\delta_n} \!re^{-n(v_{n,j}(r)-v_{n,j}(r_{\tau(j)}) - (1-\para_2)(Q_n - \check{Q}_n)(r))} \, dr \\
		&=
		 \int_{r_1+\epsilon_n(x)}^{r_1+\delta_n} \!e^{-\frac{2n^2}{\rho^2}((r-r_{\tau(j)})^2 - (1-\para_2)(r-r_1)^2)} \, dr \cdot(1+o(1))\\
		&= \frac{\rho}{2n}\int_{\frac{2n}{\rho}\epsilon_n(x)}^{\infty} e^{-\frac{1}{2}((t +\rho(1-\tau(j)))^2 -(1-\para_2)t^2 )} \, dt \cdot(1+o(1))
	\end{split}
\end{align}
where $o(1)\to 0$ uniformly for $j$. 
Note that the complementary error function satisfies the asymptotic expansion: as $z \to \infty$,
\begin{equation} \label{erfc asymp}
\erfc(z) = \frac{ e^{-z^2} }{ \sqrt{\pi}\,z } \Big(1+O\big(\frac{1}{z^2}\big)\Big),
\end{equation}
see e.g. \cite[Eq.(7.12.1)]{olver2010nist}.
Using this asymptotic together with \eqref{cn cn'} and \eqref{epsilon n}, we obtain that
\begin{align}
	\begin{split}
	\int_{\frac{2n}{\rho}\epsilon_n(x)}^{\infty} e^{-\frac{\para_2}{2}\big(t +\frac{\rho}{\para_2}(1-\tau(j))\big)^2}\, dt 
	&= \sqrt{ \frac{\pi}{ 2 \, \para_2 } } \erfc\Big(  \sqrt{ \frac{\para_2}{2} } \Big( \frac{\rho}{\para_2}(1-\tau(j))+ \frac{2n}{\rho}\epsilon_n(x)  \Big)  \Big)
	\\
	&= \frac{\rho}{2n\, \para_2 \, \epsilon_n(x)} e^{-\frac{\para_2}{2}\big(\frac{2n}{\rho}\epsilon_n(x)+\frac{\rho}{\para_2}(1-\tau(j))\big)^2}\cdot(1+o(1)),
	\end{split}
\end{align}
where $o(1)\to 0$ as $N\to \infty$ uniformly for all $j$ and all $x$ in every compact subset of $\R$. This gives 
\begin{equation}
E_n(x) = \frac{\rho}{2n\, \para_2 \, \epsilon_n(x)}\sum_{j=0}^{n-1} \frac{e^{-\frac{\para_2}{2}(\frac{2n}{\rho}\epsilon_n(x)+\frac{\rho}{\para_2}(1-\tau(j)))^2 - \frac{1}{2}(1-\frac{1}{\para_2})(\rho(1-\tau(j)))^2}}{\Phi^{(\para)}(\rho(\tau(j) - \frac{1}{2}))} \cdot (1+o(1)).
\end{equation}
As $n\to \infty$, the Riemann sum with step length $n^{-1}$ converges to the following integral:
\begin{align*}
	 \frac{\rho}{n}\sum_{j=0}^{n-1} \frac{e^{-\frac{\para_2}{2}(\frac{2n}{\rho}\epsilon_n(x)+\frac{\rho}{\para_2}(1-\tau(j)))^2 - \frac{1}{2}(1-\frac{1}{\para_2})(\rho(1-\tau(j)))^2}}{\Phi^{(\para)}(\rho(\tau(j) - \frac{1}{2}))} 
	&= \int_{-\rho}^{0} \frac{e^{-\frac{\para_2}{2}(\frac{2n}{\rho}\epsilon_n(x)-
\frac{s}{\para_2})^2 - \frac{1}{2}(1-\frac{1}{\para_2})s^2} }{\Phi^{(\para)}(s + \frac{\rho}{2})} \, ds  \cdot (1+o(1))
\\
	&= e^{ -\frac{1}{2}(\sqrt{c_n}+\frac{x}{\sqrt{c_n}})^2  } \int_{-\rho}^{0} \frac{e^{ \sqrt{\frac{{c_n}}{{\para_2}}} s - \frac{1}{2}s^2 }}{\Phi^{(\para)}(s+\frac{\rho}{2})} \, ds \cdot (1+o(1)),
\end{align*}
where the convergence is uniform for $x$ in every compact subset of $\R$. Since the above integral has the asymptotic expansion 
\begin{equation}\label{asym_in}
	\int_{-\rho}^{0} \frac{e^{-\frac{1}{2}s^2 + \sqrt{\frac{c_n}{\para_2}}s}}{\Phi^{(\para)}(s+\frac{\rho}{2})} \, ds = \sqrt{\frac{\para_2}{c_n}}\frac{1}{\Phi^{(\para)}(\frac{\rho}{2})} + O(c_n^{-1}),\qquad (n\to \infty),
\end{equation}
we obtain 
\begin{align*}
	E_n(x) = \frac{n}{\rho \,c_n \Phi^{(\para)}(\frac{\rho}{2})} e^{-\frac{1}{2}c_n - x}\cdot(1+o(1)) = e^{-x}\cdot (1+o(1)),
\end{align*}
where $o(1)\to 0$ as $n\to \infty$ uniformly for $x$ in every compact subset of $\R$.
This gives the desired convergence \eqref{Sn sum lim}.
\end{proof}

We now prove the first assertion of Theorem~\ref{Thm_M}.

\begin{proof}[Proof of Theorem~\ref{Thm_M} (i)]
It is a direct consequence of Lemma \ref{S_N} that 
$\log\Prob(\omega_n \leq x)$ converges to $-e^{-x}$, i.e. $\omega_n$ converges in distribution to the Gumbel law as $n\to \infty$. 

For the case of the minimal modulus $u_n$, the distribution function has the form
\begin{equation}
\Prob(u_n \leq x) 
	= \Prob(|\zeta|_1 \geq  r_0 - \epsilon_n'(x)), 
\end{equation}
where 
\begin{equation}
\epsilon_n'(x) = \frac{\rho}{2n\sqrt{\para_1}} \Big(\sqrt{c_n'} + \frac{x}{\sqrt{c_n'}} \Big).
\end{equation}
Recall that $c_n'$ is given in \eqref{cn cn'}. 
(Here and in the rest of the proof, the notation $'$ does not denote the differentiation.)
In the same way as for the maximal modulus, the distribution function can be expressed as the product
\begin{align*} 
\Prob(u_n \leq x) = \prod_{j=0}^{n-1} \Big( 1 - \int_{ |\zeta| < r_0 - \epsilon_n'(x)} \frac{|\zeta|^{2j} e^{-nQ_n^{(\para)}(\zeta)}}{\|\zeta^j\|^2_{nQ_n^{(\para)}}}  dA(\zeta)  \Big).
\end{align*}
Thus, we need to compute the sum
\begin{equation}
E_n'(x) := \sum_{j=0}^{n-1} \int_{0}^{r_0 - \epsilon_n'(x)} \frac{2r e^{-n(v_{n,j}(r) - (1-\para_1)(Q_n - \check{Q}_n)(r))}}{\|\zeta^j\|^2_{nQ_n^{(\para)}}} dr.
\end{equation}
As in the Lemma \ref{S_N}, we obtain the following asymptotics:
\begin{align*} 
E_n'(x) &= \sum_{j=0}^{n-1} \frac{2n}{\rho \, \Phi^{(\para)}(\rho(\tau(j)-\frac{1}{2}))} \int_{r_0-\delta_n}^{r_0 - \epsilon_n'(x)}\!\! e^{-\frac{2n^2}{\rho^2}((r-r_{\tau(j)})^2 -(1-\para_1)(r-r_0)^2)} \, dr \cdot (1+o(1))\\
& = \sum_{j=0}^{n-1} \frac{1}{\Phi^{(\para)}(\rho(\tau(j)-\frac{1}{2}))} e^{-\frac{1}{2}(1-\frac{1}{\para_1})(\rho\tau(j))^2}\int_{\frac{2n}{\rho}\epsilon_n'(x)}^{\infty} e^{-\frac{\para_1}{2}(t+\frac{1}{\para_1}\rho \tau(j))^2}  \, dt \cdot (1+o(1)).
\end{align*}
Using again the asymptotic behaviour \eqref{erfc asymp} of the complementary error function, we obtain the Riemann sum approximation
\begin{align*}
E_n'(x) 
&= \frac{\rho}{2n\, \para_1 \, \epsilon_n'(x)} \sum_{j=0}^{n-1} \frac{e^{-\frac{\para_1}{2}(\frac{2n}{\rho}\epsilon_n'(x)+\frac{\rho}{\para_1}\tau(j))^2 - \frac{1}{2}(1-\frac{1}{\para_1})(\rho\tau(j))^2}}{\Phi^{(\para)}(\rho(\tau(j) - \frac{1}{2}))} \cdot (1+o(1))\\
&= \frac{1}{2\, \para_1 \, \epsilon_n'(x)} \int_{0}^{\rho}  \frac{e^{-\frac{\para_1}{2}(\frac{2n}{\rho}\epsilon_n'(x)+\frac{s}{\para_1})^2 - \frac{1}{2}(1-\frac{1}{\para_1})s^2}}{\Phi^{(\para)}(s - \frac{\rho}{2})} \, ds \cdot (1+o(1)). 
\end{align*}
The same argument as in \eqref{asym_in} gives 
\begin{equation}
E_n'(x) = \frac{n}{\rho\,c_n' \, \Phi^{(\para)}(-\frac{\rho}{2})} e^{-\frac{1}{2}c_n' - x} \cdot (1+o(1)) = e^{-x}\cdot (1+o(1)).
\end{equation}
This completes the proof.
\end{proof}

\subsection{Soft/hard edge conditions}

In this subsection, we prove the second assertion of Theorem~\ref{Thm_M}.

\begin{proof}[Proof of Theorem~\ref{Thm_M} (i)]
For the soft/hard edge ensemble, the distribution function of $|\zeta|_n$ has the same form as above \eqref{gap_prob} (cf. \cite{MR4179777}), i.e. 
\begin{equation}
	\Prob(|\zeta|_n \leq r) = \prod_{j=0}^{n-1} \Big(1-
\int_{S_n\setminus D(0,r)} \frac{|\zeta|^{2j}e^{-nQ_n^{(\infty)}(\zeta)}}{\|\zeta^j\|^2_{nQ_n^{(\infty)}}} \, dA(\zeta) \Big).
\end{equation}
Here, $Q_n^{(\infty)}$ is the soft/hard edge potential with $\para_1 = \para_2 = +\infty$, i.e. $Q_n^{(\infty)} = Q_n \cdot \1_{S_n} + \infty \cdot \1_{\C\setminus S_n}$. 

Recall that we use the rescaling $\omega_n$ in \eqref{maxmin soft/hard}. 
Thus for $x\leq 0$ we obtain
\begin{equation} \label{Pn omega hard}
	\Prob(\omega_n\leq x) = \prod_{j=0}^{n-1}\Big(1- \int_{r_1+\frac{\rho^2 x}{cn^2}}^{r_1}\frac{2r^{2j+1}e^{-nQ_n(r)}}{\|\zeta^j\|_{nQ_n^{(\infty)}}^2} \, dr \Big).
\end{equation}
Next, we shall compute  
\begin{equation} \label{sum bfRN inf}
	\sum_{j=0}^{n-1} \int_{r_1+\frac{\rho^2 x}{cn^2}}^{r_1}\frac{2r^{2j+1}e^{-nQ_n(r)}}{\|\zeta^j\|_{nQ_n^{(\infty)}}^2} \, dr = \int_{r_1+\frac{\rho^2 x}{cn^2}}^{r_1} 2r\, \bfR_n^{(\infty)}(r)\, dr.
\end{equation}
Note that the right-hand side of \eqref{sum bfRN inf} can be approximated in terms of the rescaled correlation kernel $R_n$ defined in \eqref{R_n} as 
\begin{align}
	\begin{split}
	\int_{r_1+\frac{\rho^2x}{cn^2}}^{r_1} 2r \, \bfR_n^{(\infty)}(r)\, dr 	&= \frac{n^2}{\rho^{2}}\int_{r_1+\frac{\rho^2x}{cn^2}}^{r_1} 2 R_n^{(\infty)}(n\rho^{-1}(r-1) )\, dr \cdot (1+o(1))\\
	&= \frac{2n}{\rho}\int_{\frac{n}{\rho}(r_1 - 1) +\frac{\rho}{cn}x}^{\frac{n}{\rho}(r_1-1)} R_n^{(\infty)}(s) \,ds \cdot (1+o(1)).
	\end{split}
\end{align}
Since the rescaled kernel $R_n^{(\infty)}$ converges to $R^{(\infty)}$ uniformly in every compact subset of $\R_{-}$ (Theorem~\ref{thm:onelimit}), we have
\begin{equation} \label{int x hard}
	\frac{2n}{\rho}\int_{\frac{\rho}{4}+\frac{\rho}{cn}x}^{\frac{\rho}{4}} R_n^{(\infty)}(s) \,ds = - \frac{2n}{\rho}\cdot \frac{\rho x}{cn}R^{(\infty)}(\frac{\rho}{4}) + o(1)
 =  -x + o(1).
\end{equation}
Combining this with \eqref{Pn omega hard}, \eqref{sum bfRN inf} and \eqref{int x hard}, we conclude that $\Prob(\omega_n \leq x)\to e^{x}$ as $n\to \infty$ for $x\in \R_{-}$.

For the minimal modulus $u_n$ given in \eqref{maxmin soft/hard}, we repeat the above argument and obtain
\begin{align*}
\log \Prob(u_n \leq x) 
&= - \int_{r_0}^{r_0 - \frac{\rho^2 x}{c' n^2}} 2r\, \bfR_n^{(\infty)}(r) \, dr \cdot (1+o(1)) \\
&= -\frac{2n}{\rho} \int_{-\frac{\rho}{4}}^{-\frac{\rho}{4}-\frac{\rho x}{c' n}} R_n^{(\infty)}(s) \, ds \cdot (1+o(1)) = x + o(1)
\end{align*}
uniformly for $x \in \R_{-}$ as $n\to \infty$. The proof of Theorem~\ref{Thm_M} is finished.
\end{proof}

\subsection*{Acknowledgements} It is our pleasure to thank Yacin Ameur for helpful discussions.

\bibliographystyle{abbrv}
\bibliography{RMTbib}

\end{document}